\documentclass{amsart}
\usepackage{graphicx}
\usepackage{amsmath,bm}
\usepackage{amssymb, comment}
\usepackage[all]{xy}
\usepackage{color}

\newcommand{\Hom}{\operatorname{Hom}\nolimits}
\renewcommand{\Im}{\operatorname{Im}\nolimits}

\newcommand{\Ker}{\operatorname{Ker}\nolimits}

\newcommand{\Tor}{\operatorname{Tor}\nolimits}
\newcommand{\Ext}{\operatorname{Ext}\nolimits}

\renewcommand{\H}{\operatorname{H}\nolimits}

\newcommand{\m}{\operatorname{\mathfrak{m}}\nolimits}
\newcommand{\n}{\operatorname{\mathfrak{n}}\nolimits}

\newcommand{\comments}[1]{}

\newtheorem{theorem}{Theorem}[section]

\newtheorem{lemma}[theorem]{Lemma}
\newtheorem{proposition}[theorem]{Proposition}
\theoremstyle{definition}

\theoremstyle{definition}

\theoremstyle{definition}

\theoremstyle{definition}

\theoremstyle{definition}

\theoremstyle{definition}
\newtheorem*{notation}{Notation}
\theoremstyle{definition}
\newtheorem*{setup}{Setup}
\theoremstyle{definition}

\theoremstyle{definition}
\newtheorem*{remark}{Remark}

\theoremstyle{definition}

\begin{document}

\title{A generalized Dade's Lemma for local rings}
\author{Petter Andreas Bergh \& David A.\ Jorgensen}
\address{Petter Andreas Bergh \\ Institutt for matematiske fag \\
  NTNU \\ N-7491 Trondheim \\ Norway}
\email{bergh@math.ntnu.no}
\address{David A.\ Jorgensen \\ Department of mathematics \\ University
of Texas at Arlington \\ Arlington \\ TX 76019 \\ USA}
\email{djorgens@uta.edu}


\begin{abstract}
We prove a generalized Dade's Lemma for quotients of local rings by ideals generated 
by regular sequences.  That is, given a pair of finitely generated modules over such a ring with algebraically closed residue field, we prove a sufficient (and necessary) condition for the vanishing of all higher Ext or Tor of the modules. This condition involves the vanishing of all higher Ext or Tor of the modules over all quotients by a minimal generator of the ideal generated by the regular sequence.
\end{abstract}



\thanks{Part of this work was done while we were visiting the Mittag-Leffler Institute in February and March 2015. We would like to thank the organizers of the Representation Theory program.}

\maketitle

\section{Introduction}\label{Sec:Intro}

When is a module over a ring projective? For group algebras of elementary abelian $p$-groups, a criterion was provided by Dade in \cite{Dade}: if the ground field is algebraically closed, then a finitely generated module over such an algebra is projective if and only if its restriction to all the cyclic shifted subgroups are projective. This is now known as \emph{Dade's Lemma}. In terms of truncated polynomial rings, it can be restated as follows. Let $k$ be an algebraically closed field of positive characteristic $p$, consider the truncated polynomial ring 
$$
k[X_1, \dots, X_c]/(X_1^p, \dots, X_c^p)
$$
and denote the image of the generators $X_i$ by $x_i$. Then Dade's Lemma says that a finitely generated module over this ring is projective if and only if it is projective over the subalgebra $k[\alpha_1 x_1 + \cdots + \alpha_c x_c]$ for each $c$-tuple $( \alpha_1, \dots, \alpha_c )$ in $k^c$.

Truncated polynomial rings are local complete intersections, that is, they are quotients of regular local rings by ideals generated by regular sequences. For such rings, a cohomological generalization of Dade's Lemma is proved implicitly in \cite{Avramov} and \cite{AvramovBuchweitz}: instead of providing a criterion for finite projective dimension, a criterion for vanishing of cohomology for pairs of modules is given.  

In this paper, we generalize Dade's Lemma to quotients of arbitrary local rings, and provide both a homological and a cohomological version. More precisely, let $(S,\n,k)$ be a local ring with $k$ algebraically closed, $I \subseteq S$ an ideal generated by a regular sequence, and denote by $R$ the quotient $S/I$. In our main results (Theorem \ref{Thm:DadeTor} and Theorem \ref{Thm:DadeExt}), we provide criteria for the vanishing of $\Tor^R_n(M,N)$ and  $\Ext_R^n(M,N)$ for $n \gg 0$, where $M$ and $N$ are two finitely generated $R$-modules, in terms of the vanishing of $\Tor^{S/(f)}_n(M,N)$ and $\Ext_{S/(f)}^n(M,N)$ for $n \gg 0$ and all $f\in I\setminus \n I$. The special case when $N$ is the residue field $k$ of $R$, so that vanishing of the (co)homology is equivalent to $M$ having finite projective dimension, is proved in \cite[Corollary 2.2]{HunekeJorgensenKatz}.

\section{Chain maps}\label{Sec:ChainMaps}

Throughout this paper, a \emph{local ring} is a commutative ring which is both local and Noetherian, and all modules are assumed to be finitely generated. 

Let $(S,\n,k)$ be a local ring, $f$ a non-zerodivisor in $\n^2$, and denote by $R$ the quotient $S/(f)$. Consider a (chain) complex
$$
F \colon \cdots \to F_{n+1} \xrightarrow{\partial_{n+1}} F_n \xrightarrow{\partial_n} F_{n-1} \to \cdots
$$
of finitely generated free $R$-modules. By choosing bases for the free modules, we may view the maps $\partial_i$ in $F$ as matrices with coefficients in $R$. Now lift the whole complex to $S$, and obtain a sequence
$$\widetilde{F} \colon \cdots \to \widetilde{F}_{n+1} \xrightarrow{\widetilde{\partial}_{n+1}} \widetilde{F}_n \xrightarrow{\widetilde{\partial}_n} \widetilde{F}_{n-1} \to \cdots$$
of finitely generated free $S$-modules and maps, with $\widetilde{F} \otimes_S R$ isomorphic to $F$. Thus for each integer $n$, the free $S$-module $\widetilde{F}_n$ is of the same rank as the free $R$-module $F_n$, and the map $\widetilde{\partial}_n$ can be obtained from $\partial_n$ by choosing preimages in $S$ of all its matrix entries. 

In general, the sequence $\widetilde{F}$ is not a complex, but the image of the composition $\widetilde{\partial}_{n+1} \circ \widetilde{\partial}_{n+2}$ is contained in $f \widetilde{F}_{n}$ for all $n$. Therefore, for every element $x \in \widetilde{F}_{n+2}$ there exists an element $w_x \in \widetilde{F}_{n}$ with $\widetilde{\partial}_{n+1} \circ \widetilde{\partial}_{n+2} (x) = fw_x$. Since $f$ is a non-zerodivisor in $S$, this element $w_x$ is unique, hence the assignment
$$\widetilde{t}_{n+2} \colon \widetilde{F}_{n+2}  \to  \widetilde{F}_{n}, \hspace{5mm}
x \mapsto  w_x$$
is a well-defined $S$-homomorphism for each $n$. By definition, the equality
$$\widetilde{\partial}_{n+1} \circ \widetilde{\partial}_{n+2} = f \cdot \widetilde{t}_{n+2}$$
holds for all $n$.

\begin{lemma}\label{Lem:Commutes}
For every $n$, the diagram
$$\xymatrix{
\widetilde{F}_{n+3} \ar[r]^{\widetilde{\partial}_{n+3}} \ar[d]^{\widetilde{t}_{n+3}} & \widetilde{F}_{n+2} \ar[d]^{\widetilde{t}_{n+2}} \\
\widetilde{F}_{n+1} \ar[r]^{\widetilde{\partial}_{n+1}} & \widetilde{F}_{n} }$$
commutes.
\end{lemma}

\begin{proof}
This follows from the proof of \cite[Proposition 1.1]{Eisenbud}. For an element $x \in \widetilde{F}_{n+3}$, the definition of the maps $\widetilde{t}_i$ gives
\begin{eqnarray*}
f \cdot \left ( \widetilde{\partial}_{n+1} \circ \widetilde{t}_{n+3}(x) \right ) & = & \widetilde{\partial}_{n+1} \left ( f \cdot \widetilde{t}_{n+3}(x) \right ) \\
& = & \widetilde{\partial}_{n+1} \circ \widetilde{\partial}_{n+2} \circ \widetilde{\partial}_{n+3}(x) \\
& = & \left ( f \cdot \widetilde{t}_{n+2} \right ) \circ \widetilde{\partial}_{n+3}(x) \\
& = & f \cdot \left ( \widetilde{t}_{n+2} \circ \widetilde{\partial}_{n+3}(x) \right )
\end{eqnarray*}
Since $f$ is a non-zerodivisor on $\widetilde{F}_{n}$, the result follows.
\end{proof}

Now for each integer $n$, define an $R$-homomorphism $t_{n+2} \colon F_{n+2} \to F_n$ by
$$
t_{n+2} = \widetilde{t}_{n+2} \otimes_S R
$$
Then by applying $- \otimes_S R$ to the diagrams in the lemma, we obtain a commutative diagram
$$
\xymatrix{
\cdots \ar[r] & F_{n+3} \ar[r]^{\partial_{n+3}} \ar[d]^{t_{n+3}} & F_{n+2} \ar[r]^{\partial_{n+2}} \ar[d]^{t_{n+2}} & F_{n+1} \ar[d]^{t_{n+1}} \ar[r] & \cdots \\
\cdots \ar[r] & F_{n+1} \ar[r]^{\partial_{n+1}} & F_n \ar[r]^{\partial_n} & F_{n-1} \ar[r] & \cdots }
$$
of free $R$-modules and maps. Consequently, the double sequence 
$$
t \stackrel{\text{def}}{=}  \left ( \dots, t_{n+1}, t_n, t_{n-1}, \dots \right )
$$
constitutes a degree $-2$ chain map
$$
t \colon F \to F
$$
of complexes of free $R$-modules; this is \cite[Proposition 1.1]{Eisenbud}. 

\begin{notation}
We denote by $t(S,f, \widetilde{F} )$ the chain map $F \to F$ of degree $-2$ just constructed.
\end{notation}

This is Eisenbud's original notation. It emphasizes the fact that the chain map depends on the overlying local ring $S$, the non-zerodivisor $f \in S$, and the chosen lifting $\widetilde{F} = ( \widetilde{F}_n, \widetilde{\partial}_n )$ of the complex $F$ to $S$. Note that, by \cite[Corollary 1.4]{Eisenbud}, if $\overline{F} = ( \overline{F}_n, \overline{\partial}_n )$ is another choice of lifting of $F$ to $S$, then the two chain maps $t(S,f, \widetilde{F} )$ and $t(S,f, \overline{F} )$ are homotopic.

In the rest of this section, we shall consider the situation that occurs when the ring $R$ is obtained 
from $S$ by factoring out an ideal generated by a regular sequence of length two. The above process then produces several degree $-2$ chain maps $F \to F$, and we shall link some of these to each other.

\begin{setup}
Let $(S,\n,k)$ be a local ring, $f_1,f_2$ a regular sequence in $\n^2$, and denote by $R$ the quotient $S/(f_1,f_2)$. Furthermore, let $\alpha$ be a unit in $S$, and define the following local rings:
$$T_1 = S/(f_2), \hspace{5mm} T_2 = S/(f_1), \hspace{5mm} T_{\alpha} = S/( \alpha f_1 + f_2)$$
\end{setup}

The seemingly asymmetric indexing in the presentations of the rings $T_1$ and $T_2$ is made up for by the fact that $T_1/(f_1) = R = T_2/(f_2)$. Note also that $R = T_{\alpha}/(f_1)$. We have chosen not to distinguish notationally between the element $f_i$ in $S$ and its image in the various quotient rings.

Now consider a complex
$$F \colon \cdots \to F_{n+1} \xrightarrow{\partial_{n+1}} F_n \xrightarrow{\partial_n} F_{n-1} \to \cdots$$
of free $R$-modules again. As before, lift it to $S$ to obtain a sequence
$$\widehat{F} \colon \cdots \to \widehat{F}_{n+1} \xrightarrow{\widehat{\partial}_{n+1}} \widehat{F}_n \xrightarrow{\widehat{\partial}_n} \widehat{F}_{n-1} \to \cdots$$
of finitely generated free $S$-modules and maps, with $\widehat{F} \otimes_S R$ isomorphic to $F$. The image of $\widehat{\partial} \circ \widehat{\partial}$ is contained in $f_1 \widehat{F} + f_2 \widehat{F}$, hence there exist two graded $S$-homomorphisms 
$$\widehat{t}_1 \colon \widehat{F} \to \widehat{F}, \hspace{5mm} \widehat{t}_2 \colon \widehat{F} \to \widehat{F}$$
both of degree $-2$, with 
$$\widehat{\partial} \circ \widehat{\partial} = f_1 \cdot \widehat{t}_1 +  f_2 \cdot \widehat{t}_2$$
For each $i \in \{ 1,2 \}$, apply $- \otimes_S T_i$ to the sequence $\widehat{F}$ to get a sequence
$$\widehat{F} \otimes_S T_i \colon  \cdots \to \widehat{F}_{n+1} \otimes_S T_i \xrightarrow{\widehat{\partial}_{n+1} \otimes_S T_i} \widehat{F}_n \otimes_S T_i \xrightarrow{\widehat{\partial}_n \otimes_S T_i} \widehat{F}_{n-1} \otimes_S T_i \to \cdots$$
of free $T_i$-modules and maps. This is a lifting of the original complex $F$ to $T_i$, since $(\widehat{F} \otimes_S T_i) \otimes_{T_i} R$ is isomorphic to $F$. Moreover, the map 
$$\widehat{t}_i \otimes_S T_i \colon \widehat{F} \otimes_S T_i \to \widehat{F} \otimes_S T_i$$
is a graded $T_i$-homomorphism of degree $-2$, with
\begin{eqnarray*}
( \widehat{\partial} \otimes_S T_i ) \circ ( \widehat{\partial} \otimes_S T_i ) & = & ( \widehat{\partial} \circ \widehat{\partial} ) \otimes_S T_i \\
& = & \left ( f_1 \cdot \widehat{t}_1 +  f_2 \cdot \widehat{t}_2 \right ) \otimes_S T_i  \\
& = & f_i \cdot \widehat{t}_i \otimes_S T_i  \\
& = & f_i \cdot \left ( \widehat{t}_i \otimes_S T_i \right )
\end{eqnarray*}
We then know from the first part of this section that the map
$$t( T_i, f_i, \widehat{F} \otimes_S T_i ) = ( \widehat{t}_i \otimes_S T_i ) \otimes_{T_i} R$$
becomes a chain map $F \to F$ of degree $-2$.

Next, we repeat all this, but now for the ring $T_{\alpha} = S/( \alpha f_1 + f_2)$. The sequence
$$\widehat{F} \otimes_S T_{\alpha} \colon  \cdots \to \widehat{F}_{n+1} \otimes_S T_{\alpha} \xrightarrow{\widehat{\partial}_{n+1} \otimes_S T_{\alpha}} \widehat{F}_n \otimes_S T_{\alpha} \xrightarrow{\widehat{\partial}_n \otimes_S T_{\alpha}} \widehat{F}_{n-1} \otimes_S T_{\alpha} \to \cdots$$
of free $T_{\alpha}$-modules and maps is a lifting of $F$ to $T_{\alpha}$. Here, the graded $T_{\alpha}$-homomorphism $\widehat{F} \otimes_S T_{\alpha} \to \widehat{F} \otimes_S T_{\alpha}$ of degree $-2$ of interest is $( \widehat{t}_1 - \alpha \widehat{t}_2 ) \otimes_S T_{\alpha}$, since
\begin{eqnarray*}
( \widehat{\partial} \otimes_S T_{\alpha} ) \circ ( \widehat{\partial} \otimes_S T_{\alpha} ) & = & ( \widehat{\partial} \circ \widehat{\partial} ) \otimes_S T_{\alpha} \\
& = & \left ( f_1 \cdot \widehat{t}_1 +  f_2 \cdot \widehat{t}_2 \right ) \otimes_S T_{\alpha}  \\
& = & \left ( f_1 \cdot \widehat{t}_1 - f_1 \cdot \alpha \widehat{t}_2 + ( \alpha f_1 + f_2) \cdot \widehat{t}_2 \right ) \otimes_S T_i  \\
& = & \left ( f_1 \cdot \widehat{t}_1 - f_1 \cdot \alpha \widehat{t}_2 \right ) \otimes_S T_i  \\
& = & f_1 \cdot \left ( ( \widehat{t}_1 - \alpha \widehat{t}_2 ) \otimes_S T_i \right )
\end{eqnarray*}
The map
$$t( T_{\alpha}, f_1, \widehat{F} \otimes_S T_{\alpha} ) = \left ( ( \widehat{t}_1 - \alpha \widehat{t}_2 ) \otimes_S T_{\alpha} \right ) \otimes_{T_{\alpha}} R$$
then becomes another degree $-2$ chain map $F \to F$. The following result shows that this chain map equals the obvious linear combination of the two chain maps constructed above.

\begin{proposition}\label{Prop:ChainMaps}
Let $(S,\n,k)$ be a local ring, $f_1,f_2$ a regular sequence in $\n^2$, and denote by $R$ the quotient $S/(f_1,f_2)$. Furthermore, let $\alpha$ be a unit in $S$, and define the following local rings:
$$T_1 = S/(f_2), \hspace{5mm} T_2 = S/(f_1), \hspace{5mm} T_{\alpha} = S/( \alpha f_1 + f_2)$$
Finally, take a complex $( F, \partial )$ of finitely generated free $R$-modules, and lift it to a sequence $( \widehat{F}, \widehat{\partial} )$ of free $S$-modules and maps. Then for every unit $\alpha \in S$, there is an equality
$$t( T_{\alpha}, f_1, \widehat{F} \otimes_S T_{\alpha} ) = t( T_1, f_1, \widehat{F} \otimes_S T_1 ) - \alpha t( T_2, f_2, \widehat{F} \otimes_S T_2 )$$
of chain maps $F \to F$ of degree $-2$.
\end{proposition}

\begin{proof}
This follows from the construction of the three chain maps. Let $\widehat{t}_1$ and $\widehat{t}_2$ be the degree $-2$ homomorphisms on $\widehat{F}$ with the property that 
$$\widehat{\partial} \circ \widehat{\partial} = f_1 \cdot \widehat{t}_1 +  f_2 \cdot \widehat{t}_2$$
Then 
\begin{eqnarray*}
t( T_{\alpha}, f_1, \widehat{F} \otimes_S T_{\alpha} ) & = &  \left ( ( \widehat{t}_1 - \alpha \widehat{t}_2 ) \otimes_S T_{\alpha} \right ) \otimes_{T_{\alpha}} R \\
& = & \left ( \widehat{t}_1 \otimes_S R \right ) - \alpha \left ( \widehat{t}_2 \otimes_S R \right ) \\
& = & \left ( ( \widehat{t}_1 \otimes_S T_1 ) \otimes_{T_1} R \right ) - \alpha \left ( ( \widehat{t}_2 \otimes_S T_2 ) \otimes_{T_2} R \right ) \\
& = & t( T_1, f_1, \widehat{F} \otimes_S T_1 ) - \alpha t( T_2, f_2, \widehat{F} \otimes_S T_2 )
\end{eqnarray*}
\end{proof}

\section{Dade's Lemma}\label{Sec:Dade}

Having gone through the necessary machinery on chain maps, we now start establishing the main result. As before, let $(S,\n,k)$ be a local ring, $f$ a non-zerodivisor in $\n^2$, and denote by $R$ the quotient $S/(f)$. In order to prove our first result, we need the following fact. It implies that if a complex of free $S$-modules becomes exact after it is reduced modulo $f$, then the complex itself is exact.

\begin{lemma}\label{Lem:Exact}
Suppose that
$$X \xrightarrow{p} Y \xrightarrow{q} Z$$
is a complex of $S$-modules, with $Y$ finitely generated and $f$ a non-zerodivisor on $Z$. If the reduced complex
$$X/fX \xrightarrow{\overline p} Y/fY \xrightarrow{\overline q} Z/fZ$$
is exact, then so is the original complex.
\end{lemma}

\begin{proof} 
If $y\in\Ker q$, then $y+fY\in \ker \overline q$, and so by exactness of the reduced complex, there exists an element $x+fX\in X/fX$ such that $y+fY = \overline p(x+fX)=p(x)+fY$. Then $p(x)-y \in fY$, so  there exists an element $y'\in Y$ such that $p(x)-y=fy'$. Now notice that
$$fq(y')=q(fy')=q(p(x)-y)=q(p(x))-q(y)=0$$  
Since $f$ is a non-zerodivisor on $Z$, it follows that $q(y')=0$, i.e.\ $y'\in \Ker q$.  We have shown that
$$\Ker q = \Im p + f\Ker q$$
The module $\Ker q$ is finitely generated, being a submodule of $Y$, hence $\Ker q = \Im p$ by Nakayama's Lemma.
\end{proof}

Now let $M$ be an $R$-module, and choose a free resolution
$$F \colon \cdots \to F_3 \xrightarrow{\partial_3} F_2 \xrightarrow{\partial_2} F_1 \xrightarrow{\partial_1} F_0$$
of this module, not necessarily minimal. Thus $F$ is a complex of finitely generated free $R$-modules, with nonzero homology only in degree zero, where $\H_0 ( F ) \cong M$. As in the previous section, lift the complex to a sequence
$$\widetilde{F} \colon \cdots \to \widetilde{F}_3 \xrightarrow{\widetilde{\partial}_3} \widetilde{F}_2 \xrightarrow{\widetilde{\partial}_2} \widetilde{F}_1 \xrightarrow{\widetilde{\partial}_1} \widetilde{F}_0$$
of finitely generated free $S$-modules and maps, with $\widetilde{F} \otimes_S R$ isomorphic to $F$. Let $\widetilde{t} \colon \widetilde{F} \to \widetilde{F}$ be the graded $S$-homomorphism of degree $-2$ satisfying $\widetilde{\partial} \circ \widetilde{\partial} = f \cdot \widetilde{t}$. Thus for every $n \ge 0$ there is an $S$-homomorphism $\widetilde{t}_{n+2} \colon \widetilde{F}_{n+2}  \to  \widetilde{F}_{n}$ with 
$$\widetilde{\partial}_{n+1} \circ \widetilde{\partial}_{n+2} = f \cdot \widetilde{t}_{n+2}$$
The corresponding degree $-2$ chain map $\widetilde{t} \otimes_S R \colon F \to F$, denoted by 
$t(S,f, \widetilde{F} )$, is pictured in the commutative diagram
$$
\xymatrix{
\cdots \ar[r] & F_4 \ar[r]^{\partial_4} \ar[d]^{t_4} & F_3 \ar[r]^{\partial_3} \ar[d]^{t_3} & F_2 \ar[r]^{\partial_2} \ar[d]^{t_2} & F_1 \ar[r]^{\partial_1} & F_0 \\
\cdots \ar[r] & F_2 \ar[r]^{\partial_2} & F_1 \ar[r]^{\partial_1} & F_0 }
$$
with $t_i = t(S,f, \widetilde{F} )_i$. The mapping cone $C_{t(S,f, \widetilde{F} )}$ of this chain map is the complex
$$
\xymatrix@C=50pt{
\cdots \ar[r] & F_3 \oplus F_2 \ar[r]^{\left [ \begin{smallmatrix} - \partial_3 & 0 \\ t_3 & \partial_2 \end{smallmatrix} \right ]} & F_2 \oplus F_1 \ar[r]^{\left [ \begin{smallmatrix} - \partial_2 & 0 \\ t_2 & \partial_1 \end{smallmatrix} \right ]} & F_1 \oplus F_0 \ar[r]^{\left [ \begin{smallmatrix} - \partial_1 & 0 \end{smallmatrix} \right ]} & F_0 }
$$
with $(C_{t(S,f, \widetilde{F} )})_n = F_{n-1} \oplus F_{n-2}$. 

There is a canonical way of lifting this mapping cone complex to a sequence over $S$, in terms of the lifting $\widetilde{F}$: replace everything by the appropriate map or module over $S$, and insert the element $f$ in the upper right corner of all the matrices. The following result shows that this canonical lifting is an $S$-free resolution of the module $M$. This result will appear in \cite{Steele}; we include a proof for completeness.

\begin{theorem}\label{Thm:ResOverS}
Let $(S,\n,k)$ be a local ring, $f$ a non-zerodivisor in $\n^2$, and denote by $R$ the quotient $S/(f)$. Furthermore, let $F = (F_n, \partial_n)$ be an $R$-free resolution of an $R$-module $M$, and lift it to a sequence $\widetilde{F} = ( \widetilde{F}_n, \widetilde{\partial}_n )$ of free $S$-modules and maps. Finally, for each $n \ge 0$, let $\widetilde{t}_{n+2} \colon \widetilde{F}_{n+2}  \to  \widetilde{F}_{n}$ be the $S$-homomorphism with 
$\widetilde{\partial}_{n+1} \circ \widetilde{\partial}_{n+2} = f \cdot \widetilde{t}_{n+2}$. Then the sequence
$$\xymatrix@C=50pt{
\cdots \ar[r] & \widetilde{F}_3 \oplus \widetilde{F}_2 \ar[r]^{\left [ \begin{smallmatrix} - \widetilde{\partial}_3 & -f \\ \widetilde{t}_3 & \widetilde{\partial}_2 \end{smallmatrix} \right ]} & \widetilde{F}_2 \oplus \widetilde{F}_1 \ar[r]^{\left [ \begin{smallmatrix} - \widetilde{\partial}_2 & -f \\ \widetilde{t}_2 & \widetilde{\partial}_1 \end{smallmatrix} \right ]} & \widetilde{F}_1 \oplus \widetilde{F}_0 \ar[r]^{\left [ \begin{smallmatrix} - \widetilde{\partial}_1 & -f \end{smallmatrix} \right ]} & \widetilde{F}_0 }$$
is an $S$-free resolution of $M$.
\end{theorem}

\begin{proof} 
Denote the sequence by $C( \widetilde{F}, \widetilde{t} )$, and note that it follows directly from the definition of the maps $\widetilde{t}_i$ that it is a complex. We must show that the homology of this complex is given by
$$
\H_n \left ( C( \widetilde{F}, \widetilde{t} ) \right ) \cong \left \{ 
\begin{array}{ll}
M & \text{for $n=0$} \\
0 & \text{for $n \neq 0$}
\end{array}
\right.
$$
In what follows, denote by $\mathring{\Sigma} C( \widetilde{F}, \widetilde{t} )$ the complex 
$C( \widetilde{F}, \widetilde{t} )$ shifted one degree to the left, that is, $\left(\mathring{\Sigma} C( \widetilde{F}, \widetilde{t} )\right)_n=C( \widetilde{F}, \widetilde{t} )_{n-1}$, and \emph{without} changing the sign of the differential.

For each $n \ge 0$, let $t_{n+2} = \widetilde{t}_{n+2} \otimes_S R$, and consider the corresponding degree
$-2$ chain map
$$
t(S,f, \widetilde{F} ) \colon F \to F
$$
with $t(S,f, \widetilde{F} )_i = t_i$. We see directly that its mapping cone $C_{t(S,f, \widetilde{F} )}$ is isomorphic to the complex 
$\left ( \mathring{\Sigma} C( \widetilde{F}, \widetilde{t} )  \right ) \otimes_S R$. Now consider the canonical short exact sequence 
\begin{equation}\label{ses}
0 \to F \to C_{t(S,f, \widetilde{F} )} \to F \to 0
\end{equation}
of complexes of free $R$-modules, where each $(C_{t(S,f, \widetilde{F} )})_n = F_{n-1} \oplus F_{n-2}$ sits in the middle of the canonical split short exact sequence of free $R$-modules
$$
0\to F_{n-2} \to F_{n-1}\oplus F_{n-2} \to F_{n-1} \to 0
$$ 
This short exact sequence of complexes gives rise to a long exact sequence of homology groups, and since 
$\H_n (F) =0$ for $n \ge 1$, it follows that $\H_n ( C_{t(S,f, \widetilde{F} )} ) = 0$ for $n \ge 3$. This implies, by Lemma \ref{Lem:Exact}, that $\H_n ( \mathring{\Sigma} C( \widetilde{F}, \widetilde{t} )) =0$ for $n \ge 3$, giving $\H_n ( C( \widetilde{F}, \widetilde{t} )) = 0$ for $n \ge 2$. The isomorphisms
$$
M \cong F_0/ \Im \partial_1 \simeq \left ( \widetilde{F}_0 / f \widetilde{F}_0 \right ) / \left ( ( \Im \widetilde{\partial}_1+ f \widetilde{F}_0 ) /  f \widetilde{F}_0 \right ) \cong  \widetilde{F}_0 /  ( \Im \widetilde{\partial}_1+ f \widetilde{F}_0 )
$$
show that $\H_0 ( C( \widetilde{F}, \widetilde{t} )) \cong M$, and so it remains only to show that 
$\H_1 ( C( \widetilde{F}, \widetilde{t} )) =0$.

Let $\left [ \begin{smallmatrix} a & b \end{smallmatrix} \right ]^T$ be an element in $\Ker \left [ \begin{smallmatrix} - \widetilde{\partial}_1 & -f \end{smallmatrix} \right ]$. Then $\widetilde{\partial}_1(a) = -fb$, and so $a+f \widetilde{F}_1 \in \Ker \partial_1 = \Im \partial_2$. Choose an element $u \in \widetilde{F}_2$ such that $a+f \widetilde{F}_1 = \partial_2 (-u+f \widetilde{F}_2) = - \widetilde{\partial}_2(u) + f \widetilde{F}_1$. Then $- \widetilde{\partial}_2(u) - a$ belongs to $f \widetilde{F}_1$, and therefore $-\widetilde{\partial}_2(u) - fv = a$ for some $v \in \widetilde{F}_1$. Finally, the equalities
$$
f \cdot \left ( \widetilde{t}_2(u) + \widetilde{\partial}_1(v) \right ) = \widetilde{\partial}_1 \circ \widetilde{\partial}_2(u) + \widetilde{\partial}_1(fv) =  \widetilde{\partial}_1 \left ( \widetilde{\partial}_2(u) + fv \right ) = \widetilde{\partial}_1 (-a) = fb
$$
combined with the fact that $f$ is a non-zerodivisor on $\widetilde{F}_0$ gives $\widetilde{t}_2(u) + \widetilde{\partial}_1(v) = b$. We have shown that
$$
\left [ \begin{array}{cc} - \widetilde{\partial}_2 & -f \\ \widetilde{t}_2 & \widetilde{\partial}_1 \end{array} \right ] \left [ \begin{array}{c} u \\ v \end{array} \right ] = \left [ \begin{array}{c} a \\ b \end{array} \right ]
$$
that is, that $\H_1 ( C( \widetilde{F}, \widetilde{t} )) =0$.
\end{proof}

We can now prove the main result, the generalized Dade's Lemma for local rings. Recall first that if 
$(S,\n,k)$ is a local ring and $I$ an ideal of $S$ minimally generated by $c$ elements, then 
$I / \n I$ is a $c$-dimensional $k$-vector space. The image in $I / \n I$ of any element 
$f \in I \setminus \n I$ can be completed to a basis, and lifting these elements back to $S$ gives a minimal generating set (containing $f$) for the ideal $I$. By \cite[page 52]{BrunsHerzog}, if $I$ is generated by a regular sequence $f_1,\dots,f_c$ contained in $\n^2$, then any other minimal generating set for $I$ is also a regular sequence.

Recall also that a \emph{residual algebraic closure} of $S$ (the terminology is taken from \cite{AvramovBuchweitz}) is a faithfully flat extension $S \subseteq S^{\sharp}$ of local rings with the following properties: the maximal ideal of $S^{\sharp}$ equals $\n S^{\sharp}$, and its residue field $S^{\sharp} / \n S^{\sharp}$ is algebraically closed. By \cite[App., Th\'{e}or\`{e}me 1, Corollarie]{Bourbaki}, such an extension always exists. Now suppose that $f_1,\dots,f_c$ is a regular sequence in $S$, generating an ideal $I$, and let $R = S/I$ and $R^{\sharp} = S^{\sharp} / I S^{\sharp}$. By \cite[Lemma 2.2]{AvramovBuchweitz}, the sequence is also regular in $S^{\sharp}$, and $R^{\sharp}$ is a residual algebraic closure of $R$. We denote by $\n^{\sharp}$ the maximal ideal $n S^{\sharp}$ of $S^{\sharp}$, and by $I^{\sharp}$ the ideal $I S^{\sharp}$.

\begin{theorem}\label{Thm:DadeTor}
Let $(S,\n,k)$ be a local ring, $I \subseteq S$ an ideal generated by a regular sequence $f_1,\dots,f_c$ contained in $\n^2$, and denote by $R$ the quotient $S/I$. Furthermore, let $M$ and $N$ be two finitely generated $R$-modules. Then the following are equivalent
\begin{enumerate}
\item $\Tor^R_n(M,N)=0$ for all $n \gg 0$,
\item there exists a residual algebraic closure $S^{\sharp}$ of $S$, for which 
$$\Tor^{S^{\sharp}/(f)}_n( S^{\sharp} \otimes_S M, S^{\sharp} \otimes_S N ) =0$$ for all $n \gg 0$ and all $f \in I^{\sharp} \setminus \n^{\sharp} I^{\sharp}$,
\item for every residual algebraic closure $S^{\sharp}$ of $S$, the vanishing condition in \emph{(2)} holds.
\end{enumerate}
In particular, if $k$ itself is algebraically closed, then $\Tor^R_n(M,N)=0$ for all $n \gg 0$ if and only if $\Tor^{S/(f)}_n(M,N)=0$ for all $n \gg 0$ and all 
$f\in I \setminus \n I$.
\end{theorem}

\begin{proof}  
Let $S^{\sharp}$ be a residual algebraic closure of $S$, and $R^{\sharp}$ the corresponding closure of $R$. Then $R^{\sharp}$ is faithfully flat over $R$, and 
$$\Tor^{R^{\sharp}}_n ( R^{\sharp} \otimes_R M, R^{\sharp} \otimes_R N ) \cong R^{\sharp} \otimes_R \Tor^R_n(M,N)$$
for all $n$. Therefore $\Tor^R_n(M,N) =0$ if and only if $\Tor^{R^{\sharp}}_n ( R^{\sharp} \otimes_R M, R^{\sharp} \otimes_R N )=0$. Consequently, we may assume that $k$ is algebraically closed, and prove only the last statement. 

If $c=1$, then $I = (f_1)$, and if $f$ is an element in $I \setminus \n I$ then $f = \alpha f_1$ for some unit $\alpha \in S$. Thus $S/(f_1) \cong S/(f)$, and the result follows. 

Next, suppose that $c \ge 2$. We argue by induction on $c$, starting with the case $c=2$, i.e.\ 
$I = (f_1,f_2)$. Take an element $f\in I \setminus \n I$, complete to a regular sequence $f, g$ generating the ideal $I$, and note that $R \cong T/ (g)$, where $T=S/(f)$. Let $( F, \partial )$ be an $R$-free resolution of $M$, and lift it to a sequence $( \widehat{F}, \widehat{\partial} )$ of free $S$-modules and maps, with $F \cong \widehat{F} \otimes_S R$. The machinery in Section \ref{Sec:ChainMaps} produces a 
degree $-2$ chain map
$$
t( T,g, \widehat{F} \otimes_S T) \colon F \to F
$$
which, for simplicity, we denote by just $t$. This chain map gives rise to a short exact sequence
$$
0 \to F \to C_t \to F \to 0
$$
of complexes of free $R$-modules, as in (\ref{ses}), where $C_t$ is the mapping cone of $t$. In each degree the short exact sequence of modules splits, and so when we apply $- \otimes_R N$ we obtain another short exact sequence 
$$
0 \to F \otimes_R N  \to C_t \otimes_R N \to F \otimes_R N \to 0
$$
of complexes of $R$-modules. The corresponding long exact sequence of homology groups is
$$
\cdots \to \H_{n+1}(C_t \otimes_R N) \to \Tor^R_{n}(M,N) \xrightarrow{s} \Tor^R_{n-2}(M,N) \to \H_{n}(C_t \otimes_R N) \to \cdots
$$
where the connecting homomorphism $s$ is induced by the chain map $t$, that is, $s=\Tor^R(t,N)$. Moreover, by Theorem \ref{Thm:ResOverS} and its proof, there exists a $T$-free resolution $F_T$ of $M$ with the property that the mapping cone complex $C_t$ is isomorphic to $\mathring{\Sigma} F_T \otimes_T R$. Here $\mathring{\Sigma} F_T$ denotes the complex $F_T$ shifted one degree to the left, but with no sign change on the differential. The homology of the complex $C_t \otimes_R N$ is therefore given by
\begin{eqnarray*}
\H_n \left ( C_t \otimes_R N \right ) & \cong & \H_n \left ( ( \mathring{\Sigma} F_T \otimes_T R) \otimes_R N \right ) \\ & \cong & \H_n \left ( \mathring{\Sigma} F_T \otimes_R N \right ) \\
& \cong & \Tor^T_{n-1}(M,N)
\end{eqnarray*}
for all $n$, and so the long exact homology sequence takes the form
$$
\cdots \to \Tor^T_{n+1}(M,N) \to \Tor^R_{n+1}(M,N) \xrightarrow{s} \Tor^R_{n-1}(M,N) \to \Tor^T_{n}(M,N) \to \cdots
$$
where we have denoted by $s$ the map $\Tor^R(t,N)$.

Suppose that $\Tor^R_n(M,N)=0$ for $n \gg 0$. Then we see directly from the long exact sequence that $\Tor^T_n(M,N)$ also vanishes for $n \gg 0$. Since $T = S/(f)$, and $f$ is an arbitrary element in $I \setminus \n I$, this proves one direction of the statement. Conversely, suppose that 
$\Tor^{S/(f)}_n(M,N)=0$ for all $n \gg 0$ and all $f \in I \setminus \n I$. 
Let $\alpha$ be a unit in $S$, and consider the three local rings 
$$
T_1 = S/(f_2), \hspace{5mm} T_2 = S/(f_1), \hspace{5mm} T_{\alpha} = S/( \alpha f_1 + f_2)
$$
For $i \in \{ 1,2 \}$, denote by $t_i$ the degree $-2$ chain map $t(T_i,f_i, \widehat{F} \otimes_S T_i) \colon F \to F$, and denote the degree $-2$ chain map $t(T_{\alpha},f_1, \widehat{F} \otimes_S T_{\alpha}) \colon F \to F$ by $t_{\alpha}$. The vanishing assumption implies in particular that $\Tor^T_n(M,N)=0$ for $n \gg 0$, if we replace $T$ by each of the rings $T_1$, $T_2$ and $T_{\alpha}$. Consequently, in the long exact homology sequence, the map
$$
\Tor^R_{n+1}(M,N) \xrightarrow{s} \Tor^R_{n-1}(M,N)
$$
is an isomorphism for high $n$, if we replace the map $s$ by each of the maps 
$s_1 = \Tor^R(t_1,N)$, $s_2 = \Tor^R(t_2,N)$ and $s_{\alpha} = \Tor^R(t_{\alpha},N)$. However, it follows from Proposition \ref{Prop:ChainMaps} that $s_{\alpha} = s_1 - \alpha s_2$. This implies that for every unit $\alpha \in S$, the maps $s_2 s_1^{-1}$ and
$$s_{\alpha} s_1^{-1} = 1 - \alpha s_2 s_1^{-1}$$
are isomorphisms $\Tor^R_n(M,N) \to \Tor^R_n(M,N)$ of $R$-modules for all $n \gg 0$. If $\Tor^R_n(M,N)$ is nonzero, then this is impossible: reduce the $R$-module $\Tor^R_n(M,N)$ and the isomorphisms modulo the maximal ideal in $R$. Then $s_2 s_1^{-1}$ and $1 - \alpha s_2 s_1^{-1}$ are $k$-vector space automorphisms for all $\alpha \in k$. However, since $k$ is algebraically closed, the automorphism $s_2 s_1^{-1}$ has a nonzero eigenvalue $\lambda \in k$, and then $1 - \lambda^{-1} s_2 s_1^{-1}$ cannot be an isomorphism. Therefore $\Tor^R_n(M,N)$ must vanish for $n \gg 0$.

We have now established the case $c=2$, and proceed by induction with the case $c \ge 3$. First, suppose that $\Tor^R_n(M,N)=0$ for all $n \gg 0$, and take an element $f \in I \setminus \n I$. Complete to a regular sequence $f,f_2, \dots, f_c$ generating $I$, and let $T= S(f, f_2, \dots, f_{c-2} )$. Then 
$R \cong T/(f_{c-1},f_c)$, and the established case $c=2$ gives $\Tor^{T/(f_{c-1})}_n(M,N)=0$ for $n \gg 0$. Since $T/(f_{c-1}) \cong S/(f,f_2,\dots, f_{c-1})$, it follows by induction that $\Tor^{S/(f)}_n(M,N)=0$ for $n \gg 0$. 

Conversely, suppose that $\Tor^{S/(f)}_n(M,N)=0$ for $n \gg 0$ and for all $f \in I \setminus \n I$, and consider the ring $T = S/(f_c)$. Then (the image in $T$ of) $f_1, \dots, f_{c-1}$ is a regular sequence in $T$, generating an ideal $J \subseteq T$ with $T/J \cong R$. Denote the maximal ideal of $T$ by $\m$, and let $g$ be an arbitrary element in $J \setminus \m J$. Lift this element back to $S$ and obtain an element 
$g' \in (f_1, \dots, f_{c-1}) \setminus \n (f_1, \dots, f_{c-1})$, and note that the sequence $f_c, g'$ is regular. Now take any element $f$ in $(f_c, g') \setminus \n (f_c, g')$. Then $f$ also belongs to $I \setminus \n I$, hence $\Tor^{S/(f)}_n(M,N)=0$ for $n \gg 0$. The case $c=2$ then gives $\Tor^{S/(f_c,g')}_n(M,N)=0$ for $n \gg 0$. Since $S/(f_c,g') \cong T/(g)$, this implies that $\Tor^{T/(g)}_n(M,N)=0$ for $n \gg 0$. As $g$ was an arbitrary element in $J \setminus \m J$, it now follows from the induction hypothesis that $\Tor^R_n(M,N)=0$ for $n \gg 0$. This completes the proof.
\end{proof}

What is the connection between this result and Dade's Lemma, i.e.\ \cite[Lemma 11.8]{Dade}? Let $E$ be a finite elementary abelian $p$-group and $k$ an algebraically closed field of characteristic $p$. Then Dade's Lemma says that a finitely generated kE-module is projective if and only if its restriction to each of the cyclic shifted subgroups of $E$ is projective. Now recall that the group algebra $kE$ is isomorphic to the truncated polynomial ring
$$
k[X_1, \dots, X_c]/(X_1^p, \dots, X_c^p)
$$
where $c$ is the rank of $E$. For simplicity, we identify $kE$ with this ring. Let $x_i$ denote the 
coset $X_i+(X_1^p, \dots, X_c^p)$ in $kE$. Given a $c$-tuple $\alpha = ( \alpha_1, \dots, \alpha_c )$ in $k^c$, denote the element $\alpha_1x_1 + \cdots + \alpha_c x_c$ in $kE$ by $x_{\alpha}$. Then $x_{\alpha}^p =0$, and so the element $x_{\alpha}$ corresponds to an element in $E$ generating a cyclic shifted subgroup of $E$. The subalgebra $k[x_{\alpha}]$ of $kE$ generated by $x_{\alpha}$ is isomorphic to the truncated polynomial ring $k[y]/(y^p)$. 

When we identify $kE$ and the truncated polynomial ring as above, Dade's Lemma takes the following form: a finitely generated $kE$-module is projective if and only if it is projective over the subalgebra 
$k[x_{\alpha}]$ for every $\alpha \in k^c$. However, a module is projective over $k[x_{\alpha}]$ if and only if it is projective over $k[X_1, \dots, X_c]/(\alpha_1^pX_1^p+\cdots +\alpha_c^pX_c^p)$; this is implicit in the proof of \cite[Theorem 7.5]{Avramov}, and proved explicitly in \cite[Lemma 4.3]{BerghJorgensen}. Thus with $S = k[X_1, \dots, X_c]$, the regular sequence being $X^p_1, \dots, X^p_c$, and $N=k$, Theorem \ref{Thm:DadeTor} above is equivalent to Dade's Lemma.

\begin{remark}
The change of rings long exact homology sequence
$$\cdots \to \Tor^T_{n+1}(M,N) \to \Tor^R_{n+1}(M,N) \xrightarrow{s} \Tor^R_{n-1}(M,N) \to \Tor^T_{n}(M,N) \to \cdots$$
in the proof of Theorem \ref{Thm:DadeTor} is standard and appears many places in the literature (here $R = T/(g)$ and $g$ is a non-zerodivisor in $T$). It is nothing but the degenerated change of rings spectral sequence
$$\Tor^R_p \left ( M, \Tor^T_q(N,R) \right ) \underset{p}{\Longrightarrow} \Tor^T_{p+q}(M,N)$$
but can also be derived in a completely elementary way, as in \cite[proof of Lemma 1.5]{Murthy}. However, in order to obtain an explicit description of the connecting homomorphism
$$\Tor^R_{n+1}(M,N) \xrightarrow{s} \Tor^R_{n-1}(M,N)$$
we have to involve techniques as those in Section \ref{Sec:ChainMaps}. The cohomology version is extensively treated in \cite[Sections 1 and 2]{Avramov}.
\end{remark}

We end this paper with the cohomology version of Theorem \ref{Thm:DadeTor}. Its proof is a slight variation of that of the homology version. This result also follows from \cite[Proposition 2.4 and Theorem 2.5]{AvramovBuchweitz}, where the machinery of support varieties is utilized. The results of \cite{AvramovBuchweitz} cannot however be used to deduce the homology version, Theorem \ref{Thm:DadeTor}, as the theory of support varieties therein does not pertain to Tor.

\begin{theorem}\label{Thm:DadeExt}
Let $(S,\n,k)$ be a local ring, $I \subseteq S$ an ideal generated by a regular sequence $f_1,\dots,f_c$ contained in $\n^2$, and denote by $R$ the quotient $S/I$. Furthermore, let $M$ and $N$ be two finitely generated $R$-modules. Then the following are equivalent
\begin{enumerate}
\item $\Ext_R^n(M,N)=0$ for all $n \gg 0$,
\item there exists a residual algebraic closure $S^{\sharp}$ of $S$, for which 
$$\Ext_{S^{\sharp}/(f)}^n( S^{\sharp} \otimes_S M, S^{\sharp} \otimes_S N ) =0$$ for all $n \gg 0$ and all $f \in I^{\sharp} \setminus \n^{\sharp} I^{\sharp}$,
\item for every residual algebraic closure $S^{\sharp}$ of $S$, the vanishing condition in \emph{(2)} holds.
\end{enumerate}
In particular, if $k$ itself is algebraically closed, then $\Ext_R^n(M,N)=0$ for all $n \gg 0$ if and only if $\Ext_{S/(f)}^n(M,N)=0$ for all $n \gg 0$ and all 
$f\in I \setminus \n I$.
\end{theorem}

\begin{proof}
The proof follows that of Theorem \ref{Thm:DadeTor}. In each degree, the short exact sequence
$$0 \to F \to C_t \to F \to 0$$
of complexes of free $R$-modules splits, and so when we apply $\Hom_R(-,N)$, the result is again a short exact sequence
$$0 \to \Hom_R(F,N) \to \Hom_R(C_t,N) \to \Hom_R(F,N) \to 0$$
of complexes. The complex $C_t$ is isomorphic to $\mathring{\Sigma} F_T \otimes_T R$, where $F_T$ is a $T$-free resolution of $M$. Therefore, in the resulting long exact cohomology sequence, the cohomology of the complex $\Hom_R(C_t,N) $ is given by
\begin{eqnarray*}
\H^n \left ( \Hom_R(C_t,N) \right ) & \cong & \H^n \left (  \Hom_R(\mathring{\Sigma} F_T \otimes_T R,N) \right ) \\ 
& \cong & \H^n \left (  \Hom_T( \mathring{\Sigma} F_T, \Hom_R(R,N) \right ) \\ 
& \cong & \H^n \left (  \Hom_T( \mathring{\Sigma} F_T, N) \right ) \\
& \cong & \Ext_T^{n-1}(M,N)
\end{eqnarray*}
by adjointness. Consequently, we obtain a long exact sequence 
$$
\cdots \to \Ext_T^{n-1}(M,N) \to \Ext_R^{n-2}(M,N) \xrightarrow{u} \Ext_R^{n}(M,N) \to \Ext_T^{n}(M,N) \to \cdots
$$ 
The rest of the proof of Theorem \ref{Thm:DadeTor} now carries over.
\end{proof}

\end{document}